\documentclass[12pt]{amsart}
\usepackage{amsmath}
\usepackage{amsthm}
\usepackage{amssymb}
\usepackage{amscd, mathrsfs}

\usepackage[colorlinks]{hyperref}
\usepackage[all,dvips]{xy}

\newtheorem{theorem}{Theorem}[section]
\newtheorem{lemma}[theorem]{Lemma}

\newtheorem{example}[theorem]{Example}

\newtheorem{prop}[theorem]{Proposition}
\newtheorem{cor}[theorem]{Corollary}
\newtheorem{remark}[theorem]{Remark}

\numberwithin{equation}{section}

\newcommand{\coker}{\mathop{\mathrm{Coker}}\nolimits}

\newcommand{\abs}[1]{\lvert#1\rvert}

\newcommand{\Z}{{\mathbb Z}}

\newcommand{\spec}{\mathop{\mathrm{Spec}}\nolimits}

\newcommand{\rank}{\mathop{\mathrm{rank}}\nolimits}

\newcommand{\proj}{\mathop{\mathrm{Proj}}\nolimits}

\renewcommand{\Z}{\mathbb Z}

\newcommand{\longto}{\longrightarrow}
\renewcommand{\l}{\ell}
\newcommand{\8}{{\infty}}
\newcommand{\T}{{\lceil}}
\newcommand{\4}{{\lfloor}}
\newcommand{\7}{{\rceil}}
\newcommand{\3}{{\rfloor}}
\newcommand{\F}{\mathcal{F}}

\renewcommand{\O}{\mathcal{O}}

\renewcommand{\S}{\mathcal{S}}
\newcommand{\V}{\mathcal{V}}
\newcommand{\W}{\mathcal{W}}

\newcommand{\syz}{\mathop{\mathrm{Syz}}\nolimits}

\renewcommand{\le}{\leqslant}

\newcommand{\lra}{\longrightarrow}

\renewcommand{\l}{\ell}

\newcommand{\HH}{\mathrm{H}}
\newcommand{\h}{\mathrm{h}}

\begin{document}

\title[A Direct Limit for Limit Hilbert-Kunz Multiplicity]{A Direct Limit for Limit Hilbert-Kunz Multiplicity for Smooth Projective Curves}

\author{Holger Brenner}
\address{Fachbereich f\"ur Mathematik und Informatik, Universit\"at Osnabr\"uck, Osnabr\"uck, Germany}
\email{hbrenner@uni-osnabrueck.de}

\author{Jinjia Li}
\address{Mathematics Department, Syracuse University, Syracuse, NY 13224, USA}
\curraddr{Department of Mathematics, University of Louisville, Louisville, KY 40292,  USA} 
\email{jinjia.li@louisville.edu}

\thanks{The third author was supported in part by NSA Grant \#H98230-08-1-0025.}
\author{Claudia Miller}
\address{Mathematics Department, Syracuse University, Syracuse, NY 13224, USA} \email{clamille@syr.edu }

\date{\today}

\begin{abstract}
This paper concerns the question of whether a more direct limit can be used to obtain the limit Hilbert-Kunz multiplicity, a possible candidate for a characteristic zero Hilbert-Kunz multiplicity. The main goal is to establish an affirmative answer for one of the main cases for which the limit Hilbert-Kunz multiplicity is even known to exist, namely that of graded ideals in the homogeneous coordinate ring of smooth projective curves. The proof involves more careful estimates of bounds found independently by Brenner and Trivedi on the dimensions of the cohomologies of twists of the syzygy bundle as the characteristic $p$ goes to infinity and uses asymptotic results of Trivedi on the slopes of Harder-Narasimham filtrations of Frobenius pullbacks of bundles. In view of unpublished results of Gessel and Monsky, the case of maximal ideals in diagonal hypersurfaces is also discussed in depth.
\end{abstract}



\maketitle

\section*{Introduction}

In 1983, following Kunz's lead in \cite{Ku}, 
Monsky defined in  \cite{Mo-original} a new multiplicity in positive characteristic 
-- the Hilbert-Kunz (HK) multiplicity -- as follows: Let $R$ be a
ring of characteristic $p>0$ and $I=(f_1, \dots, f_s)$ an ideal with
the length $\ell(R/I)$ finite. Consider the Frobenius powers
$I^{[p^n]}=(f_1^{p^n}, \dots, f_s^{p^n})$ of $I$ and define
\[
e_{\textrm{HK}}(I,R) = \lim_{n\to\8} \frac{\ell(R/I^{[p^n]})}{(p^n)^{\dim (R)}}
\]
Just like the usual Hilbert-Samuel multiplicity, this new multiplicity 
seems to measure the degree of singularity at a point on a variety. 
Furthermore, it plays the role for tight closure that ordinary
Hilbert-Samuel multiplicity plays for integral closure. 
But the numbers seem much more complex (they are
usually not integers and possibly not always rational or even algebraic) than usual
multiplicities (which are integers) and, despite intense study in
recent years, are still not well understood or even computable except
in a few cases. 

However what little is known seems to indicate that
the numbers may get simpler in the limit as the characteristic $p$
goes to infinity, leading to the question of whether a characteristic
zero HK multiplicity defined in such a way could have a more
transparent meaning or behavior than the one in characteristic
$p$ does.
More precisely, if $R$ is a $\mathbb Z$-algebra essentially of finite type over $\mathbb Z$ 
and $I$ an ideal, let 
$R_p$ be the reduction of $R$ mod $p$ and $I_p$ the extended ideal. 
If $\ell(R_p/I_p)$ is finite and nonzero for almost all $p$, define 
\[
e_{\textrm{HK}}^{\8}(I,R) 
\stackrel{\textrm{def}}{=}  
 \lim_{p\to\8} e_{\textrm{HK}}(I_p,R_p) 
\]
whenever this limit exists, and call it the {\it limit Hilbert-Kunz multiplicity} of $I$.  

Although experimental results indicate this limit might 
always exist, very few cases have been established.  
It is, of course, clear when $e_{\textrm{HK}}(I_p,R_p)$ 
is constant for almost all $p$, such as for the homogeneous maximal ideal 
in the coordinate rings of plane cubics \cite{BC}, \cite{Mo-cubic}, \cite{P}, 
in certain monomial ideals \cite{Bruns}, \cite{Co}, \cite{E}, \cite{W}, 
in two-dimensional invariant rings under finite group actions \cite{WY}, 
and for full flag varieties and for elliptic curves embedded by complete linear systems \cite{FT} 
(see also \cite{BH}).   
That this is also the case for ideals of finite projective dimension
can be seen via local Riemann-Roch theory (private communication with Kurano); 
it is interesting that in this last case the limit has an 
intrinsic geometric interpretation in characteristic zero. 
A few nonconstant cases are known as well: 
The limit was shown to exist for the homogeneous maximal ideal of diagonal hypersurfaces, in 
unpublished work of Gessel and Monsky \cite{Mo-unpub} building on \cite{HM}. 
It was also shown to exist for any homogeneous ideal primary to the 
homogeneous maximal ideal in homogeneous coordinate rings of smooth projective curves by 
Trivedi in \cite{Tr2} by delicate study of the variation of Harder-Narasimhan filtrations 
of Frobenius pullbacks of the syzygy bundle relative to the characteristic $p$.
The limit in this case turns out again to have an intrinsic geometric description 
in characteristic zero. 


In this paper, we are interested in the question of whether a simpler limit gives 
the same result. In particular, is it necessary to use the full HK multiplicity 
$e_{\textrm{HK}}(I_p,R_p)$  in each characteristic $p$? This value is 
itself the usually uncomputable limit 
${{\lim_{n\to\infty} \frac{\ell(R_p/I_p^{[p^n]})}{(p^n)^d}}}$ 
where $d=\dim R$. 
We propose to replace this complex limit with its first term $\frac{\ell(R_p/I_p^{[p]})}{p^d}$ 
or more generally any fixed degree term as follows:

{\bf{Question.}} \ 
Assuming $e_{\textrm{HK}}^{\8}(I,R)$ exists, is it true that for any fixed ${n} \geq 1$
\[
e_{\textrm{HK}}^{\8}(I,R) = 
\lim_{p\to\8} \frac{\ell(R_p/I_p^{[p^{n}]})}{(p^{n})^{d}}
\ \ ? 
\]

Informally, in measuring colengths of $p^n$th bracket powers of the ideal, if $p$ goes to
infinity, is it really necessary to first let $n$ go to infinity?

The motivation behind such a modification is that a simpler limit 
may make it easier to find a geometric interpretation of 
the limit HK multiplicity in characteristic zero. It would be encouraging 
to see a simpler limit giving the possible characteristic zero concept. 
A drawback is that it still does not yield an intrinsic definition of 
$e_{\textrm{HK}}^{\8}(I,R)$ in a characteristic zero setting.

The main goal of this paper is to establish an affirmative answer to the question 
for the case of the homogeneous coordinate rings of smooth projective curves.
Our proof is based on the proofs in this setting of Brenner \cite{Br}
and Trivedi \cite{Tr1, Tr3} of a formula for the HK multiplicity and of
Trivedi \cite{Tr2} regarding the existence of
$e_{\textrm{HK}}^{\8}(I,R)$, but requires some additional work 
as we may not assume that the fixed value $n$ is large 
enough to give strong Harder-Narasimham filtrations of the syzygy bundle 
(the case $n=1$ is the most important).
Fortunately, the gap can be filled using Trivedi's results mentioned above 
to yield: 

\noindent {\bf{Corollary~\ref{maincorollary}}} 
{\emph{Let $R$ be a standard-graded flat domain over $\mathbb Z$ such that 
almost all fiber rings $R_p=R \otimes_{\mathbb Z}{\mathbb Z}/p{\mathbb Z}$ 
are geometrically normal 2-dimensional domains and let $I=(f_1, \ldots, f_s)$ 
be a homogeneous $R_+$-primary ideal. 
With the notation as above, for any fixed $n \geq 1$ one has 
\[
\frac{\l(R_p/I_p^{[p^n]})}{(p^n)^2}
 = e_{\text{HK}}^{\8}(I,R)  + O\left(\frac{1}{p}\right)
\] 
}}

We remark that, with this result, 
the answer to the question above is known to be yes 
in all the main cases in which $e_{\textrm{HK}}^{\8}(I,R)$ is
known to exist so far.  

Section~\ref{background} contains a review of the background.
The groundwork for our main result is done in Section~\ref{bundles} via 
some lemmas on the asymptotic growth of cohomologies of bundles as 
the characteristic $p$ goes to infinity. 
In Section~\ref{mainresult} these lemmas are applied to the syzygy 
bundle, defined in (\ref{def-syzygy}), to obtain the corollary above. 

The remaining part, Section~\ref{monsky}, is devoted to a discussion 
of consequences of Gessel and Monsky's unpublished work \cite{Mo-unpub}. 
We see that a side-product of their proof is an affirmative answer to the 
question above for the case of diagonal hypersurfaces. 
Furthermore, their work shows that the most tempting naive limit in 
characteristic zero does not give $e_{\textrm{HK}}^{\8}(I,R)$. 

Finally, we mention our convention regarding asymptotics throughout the paper: 
Let $q=p^n$. We emphasize that for the asymptotic notation $O(-)$ used
throughout the paper, such as in $O(\frac{q^2}{p})$, $O(q)$, or even
$O(1)$, we have fixed $n > 0$ and let $p \to \8$ (unlike in \cite{Br}
and \cite{Tr1}, where $p$ is fixed and $n$ is allowed to go to infinity).

\section{Preliminaries and Background}
\label{background}

In this section, we present the basic set-up and notations and review
relevant results on vector bundles.

{\bf{Basic set-up}}

Let $R$ be a standard-graded flat domain over $\mathbb Z$ such that 
almost all fiber rings $R_p=R \otimes_{\mathbb Z}{\mathbb Z}/p{\mathbb Z}$ 
are geometrically normal 2-dimensional domains. Let $I=(f_1, \ldots, f_s)$ 
be a homogeneous $R_+$-primary ideal with $\deg f_i=d_i$. 
Let $Y=\proj R_{\mathbb Q}$ where 
$R_\mathbb Q= R \otimes_{\mathbb Z} {\mathbb Q}$. For
each prime $p$, consider the reduction to characteristic $p$
\[
R_p=R \otimes _{\mathbb Z}{\mathbb Z}/p{\mathbb Z}, \text{\ \ \ } 
I_p=IR_p, \text{\ \ \ } 
Y_p= \proj R_p 
\]
Due to our assumptions, $Y$ and $Y_p$ are smooth projective curves for almost all $p$. The corresponding Hilbert-Kunz multiplicity is
\[
e_{\text{HK}}(I_p,R_p)\overset{\textrm{ def }}{=} 
\lim_{n \to \8} \dfrac{\l(R_p/I_p^{[q]})}{q^2}
\]
where $q=p^n$. The key idea in \cite{Br} and \cite{Tr1} for
determining the Hilbert-Kunz multiplicity is to consider the syzygy
bundle $\S=\syz(f_1, \dots, f_s)$ on $Y_p$ (and on $Y$) given by
\begin{equation}
\label{def-syzygy}
0\lra \S \lra \bigoplus_{i=1}^s \O(-d_i) 
\overset{f_1, \ldots, f_s} {\longto} \O \lra 0
\end{equation} 
and the pullback of this exact sequence $n$ times along the 
absolute Frobenius morphism $F: Y_p \lra Y_p$ (with a subsequent
twist by $m\in \mathbb Z$)
\begin{equation}
\label{exact-seq}
0\lra \S^q(m) \lra \bigoplus_{i=1}^s \O(m-qd_i) 
\overset{f_1^q, \dots, f_s^q} {\lra} \O(m) \lra 0
\end{equation} 
where $\S^q$ denotes the pullback $(F^*)^n(\S) = \syz (f_1^q, \dots, f_s^q)$. 

\begin{remark}
\label{syzygy-notation}
Notice that for simplicity, we use the
notation $\S$ for the syzygy bundle over any $Y_p$, as the
characteristic is usually obvious from the context (we
study mostly $\S^q$, not $\S$). The first sequence is just a reduction mod $p$ of the corresponding sequence in characteristic zero. In particular, $\S$ is the reduction to $Y_p$ of the syzygy bundle on $Y$.
\end{remark}

As $R_p$ is normal, the cokernel of the second map
in the associated long exact sequence of cohomology
\[
0\lra H^0(Y_p, \S^q(m)) \lra \bigoplus_{i=1}^s H^0(Y_p, \O(m-qd_i)) \stackrel{f_1^q
, \ldots, f_s^q} {\longto} H^0(Y_p, \O(m)) \lra \cdots
\] 
is the $m$-th graded piece of $R_p/I_p^{[q]}$. 
Brenner \cite{Br} and Trivedi \cite{Tr1} exploited this connection to
$H^0(Y_p,\S^q(m))$ to determine the Hilbert-Kunz multiplicity of
$I_p$ in terms of intrinsic properties of the syzygy bundle, which
we review next.

{\bf{Harder-Narasimhan filtrations}}

Let $X$ be a smooth projective curve over an algebraically closed
field. For any vector bundle $\V$ on $X$ of rank $r$, the \emph{degree}
and \emph{slope} are defined respectively as
\[
\deg(\V)\overset{\textrm{ def }}{=}
\deg(\wedge^r \V) \hspace{3cm} \mu(\V)\overset{\textrm{ def }}{=}
\dfrac{\deg(\V)}{r}
\]
Slope is additive on tensor products of bundles: $\mu(\V \otimes \W)=
\mu(\V)+\mu(\W)$. If $f: X'\lra X$ is a finite map of degree $q$,
then $\deg (f^*(\V))=q\deg(\V)$ and so $\mu(f^*(\V))=q\mu(\V).$

A bundle $\V$ is called \emph{semistable} if for every subbundle
$\W\subseteq \V$ one has $\mu(\W) \leq \mu(\V)$. Clearly,
bundles of rank 1 are always semistable, and duals and twists of
semistable bundles are semistable.

Any bundle $\V$ has a filtration by subbundles
\[
0=\V_0 \subset \V_1 \subset \cdots \subset \V_t=\V
\]
such that $\V_k/\V_{k-1}$ is semistable and $\mu(\V_k/\V_{k-1}) >
\mu(\V_{k+1}/\V_k)$ for each $k$. This filtration is unique, and
it is called the \emph{Harder-Narasimhan (or HN) filtration} of $\V$.

The \emph{maximal} and \emph{minimal slopes} are defined as
\[
\mu_{max}(\V)\overset{\textrm{ def }}{=}
\mu(\V_1/\V_0) 
\hspace{3cm} 
\mu_{min}(\V)\overset{\textrm{ def }}{=}
\mu(\V_t/\V_{t-1}) 
\]

\begin{remark}
In positive characteristic, pulling back under the Frobenius morphism $F$ 
does not necessarily preserve semistability. Therefore, the
pullback under $F^n$ of an HN filtration of
$\V$ does not always give an HN filtration of
$(F^*)^n(\V)$. The existence of a strong HN filtration from \cite{La} 
was crucial to the work in \cite{Br} and \cite{Tr1}, 
i.e., for some $n_0$, the HN filtration of $(F^*)^{n_0}(\V)$ has the property that all its
Frobenius pullbacks are the HN filtrations of $(F^*)^n(\V)$, for all
$n>n_0$. 
\end{remark}

We do not need strong HN filtrations here since for us 
$n$ is fixed at a given value and cannot be modified, but we do need
some relation between the HN filtrations of $\S$ and $\S^q$.
Fortunately, for $p\gg 0$, the following refinement result by
Trivedi {\cite[Lemmas 1.8 and 1.14]{Tr2}} applies:

\begin{prop}[Trivedi]
\label{Trivedi-lemma}  
Let $\V$ be a bundle of rank $r$ on a
smooth projective curve $X$ of genus $g$ over an algebraically closed
field of characteristic $p$ with $p>4(g-1)r^3$. Let $n\geq 1$ and $q=p^n$.
If
\[
0=\V_0 \subset \V_1 \subset \cdots \subset \V_t=\V
\]
is the HN filtration of $\V$, then its pullback 
\[
0=(F^*)^n(\V_0) \subset (F^*)^n(\V_1) \subset \cdots
 \subset (F^*)^n(\V_t)=(F^*)^n(\V)
\]
can be refined to the HN filtration of $(F^*)^n(\V)$.

Furthermore, denoting the $k$th portion of the refined filtration as follows
\[
(F^*)^n(\V_{k-1})=\V_{k,0}\subset \V_{k,1} \subset \cdots 
\subset \V_{k,t_k}=(F^*)^n(\V_k)
\]
one has that for any $i$ 
\[
\left|
\dfrac{\mu(\V_{k,i}/\V_{k,i-1})}{q}  -  \mu(\V_k/\V_{k-1}) 
\right| 
\leq  \dfrac{C}{p}
\]
where $C$ is a constant depending only on $g$ and
$r$.
\end{prop}

In our situation the curves $Y$ and $Y_p$ are not defined over an algebraically closed field, but due to our assumptions the curves
$Y_{\overline{\mathbb Q}} = Y \times_{\mathbb Q} \overline{\mathbb Q}$ and
$\overline{Y}_p= Y_p \times_{{\mathbb Z}/p{\mathbb Z}} {\overline{{\mathbb Z}/p{\mathbb Z}}}$ are smooth projective curves over the algebraic closures. In our setting the definition of degree, semistability and the Harder-Narasimhan filtration descends to the original curves. Hence we will move to the algebraic closure and back whenever this is convenient. Moreover, because of the openness of semistability in a family, the Harder-Narasimhan filtration of $\S$ on $Y$ extends to the Harder-Narasimhan filtration almost everywhere, so that the slopes of the quotients in the Harder-Narasimhan filtration of $\S$ on $Y_p$ are constant for almost all $p$.

\section{Asymptotic Lemmas for Bundles}
\label{bundles}

In this section we prove various asymptotic results on the cohomologies of
bundles that will be used in the next section for the proof of the main result. 
Let $\S$ be any bundle on the relative curve $\proj R \rightarrow \spec \Z$. Fix $n\geq 0$ and set $q=p^n$ for varying $p$.  We denote the restriction of $\S$ to $Y_p$ again by the
symbol $\S$, as this should cause no confusion in context. We first
review the notation that we use to describe concisely the data from the
various HN filtrations.

{\bf{Notation}}

We continue this practice of introducing notation unadorned
by the characteristic $p$ as it will always be obvious from the
context.

First, for each $p$, write the HN filtration of $\S$ as
\[
0=\S_0 \subset \S_1 \subset \cdots \subset \S_t=\S
\]
with slopes, normalized slopes, and ranks (for $k=1,\ldots,t$) defined as follows:
\[
\mu_k\overset{\textrm{ def }}{=}\mu(\S_k/\S_{k-1}) 
\hspace{1.3cm} 
\nu_k\overset{\textrm{ def }}{=}\dfrac{-\mu_k}{\deg Y_p}
\hspace{1.3cm} 
r_k\overset{\textrm{ def }}{=}\rank(\S_k/\S_{k-1}) 
\]

Throughout we will assume that 
$p$ has been taken to be large enough so that the notations
$\mu_k$, $\nu_k$ and $r_k$ refer to constants.

Taking pullbacks under the $n$th Frobenius morphism and setting
\[
\S_k^q\overset{\textrm{ def }}{=} (F^*)^n(\S_k)
\]
gives
\[
0=\S_0^q \subset \S_1^q \subset \cdots \subset \S_t^q=\S^q
\]
By Proposition~\ref{Trivedi-lemma}, for $p \gg 0$, the HN filtration of $\S^q$ 
can be obtained by refining each containment above, say as
\[
S_{k-1}^q=\S_{k,0} \subset \S_{k,1} \subset \cdots \subset \S_{k,t_k}=\S_k^q
\]
We denote the maximal and minimal slopes in this portion as
\[
\mu_k^{max}\overset{\textrm{ def }}{=}\mu(\S_{k,1}/\S_{k,0}) 
\text{\ \ and \ \ }
\mu_k^{min}\overset{\textrm{ def }}{=}\mu(\S_{k,t_k}/\S_{k,t_k-1}) 
\]
(we will not need the intermediate slopes). 
Further, we define normalized versions of these slopes as
\[
\nu_k^{max}\overset{\textrm{ def }}{=}\dfrac{-\mu_k^{max}}{q\deg Y_p} 
\text{\ \ and \ \ } 
\nu_k^{min}\overset{\textrm{ def }}{=}\dfrac{-\mu_k^{min}}{q\deg Y_p}
\]
Note that
\[
\mu_1^{max}\geq\mu_1^{min}>\mu_2^{max}\geq\mu_2^{min}>\cdots>
\mu_k^{max}\geq\mu_k^{min}> \cdots >\mu_t^{max}\geq\mu_t^{min}
\]
and therefore
 \[
 \nu_1^{max}\leq\nu_1^{min}<\nu_2^{max}\leq\nu_2^{min}<\cdots< 
 \nu_k^{max}\leq\nu_k^{min}< \cdots <\nu_t^{max}\leq\nu_t^{min}
 \]
 In this situation, Trivedi's result, Proposition~\ref{Trivedi-lemma}, becomes:
 
\begin{cor}[Trivedi]
 \label{Trivedi-lemma-simple}
 With the notations as above, for any $k$, as $p \to \8$
 \[
 \nu_k^{max}=\nu_k+O\left(\dfrac{1}{p}\right) 
 \text{\ \ and \ \ }  
 \nu_k^{min}=\nu_k+O\left(\dfrac{1}{p}\right)
 \]
 \end{cor}
 
Furthermore, letting $\omega$ denote the canonical bundle, we set 
\[
\theta=\frac{\deg \omega_{Y_p}}{\deg Y_p}
\]
which is constant for $p\gg 0$ by the earlier discussion. 

Lastly, for any sheaf $\F$ on $Y_p$ 
we write $h^i(\F)$ or $h^i(Y_p, \F)$ for $\dim_k H^i(Y_p,\F)$.

{\bf{Asymptotic Lemmas}}

We first prove a lemma on the cohomology of the twisted bundles
$\S^q(m)$ in various ranges of $m$. Both the lemma and its proof are
in direct analogy with Proposition~3.4 of \cite{Br}, but as now we
have that $p$, not $n$, is going to infinity, some more care must be
taken. In particular, note that we cannot use strong HN
filtrations as $n$ is fixed. Instead we compare the filtration to that
of the original bundle using the results of Trivedi described in
Section~\ref{background}.

In the proofs of the asymptotic parts of the next few results, we
assume that $p$ has been taken large enough so that the genus and
degree of $Y_p$ equal those of $Y$, and we denote them by $g$ and
$\deg Y$, respectively.  We also assume that $p$ is large enough so that 
the slopes $\mu_k$ and normalized slopes $\nu_k$ are constant and 
that $\deg \omega_{Y_p}=\deg \omega_Y$. 

Note that for $p\gg 0$ one has inequalities 
\[
q\nu_k^{max} \le q\nu_k^{min} < q\nu_k^{min} + \theta < q\nu_{k+1}^{max}
\]
where the last one holds by Corollary~\ref{Trivedi-lemma-simple}, 
the inequality $\nu_k  < \nu_{k+1}$  and the fact that $\theta$ is constant for $p\gg 0$.

\begin{lemma}
\label{lemma123}
Let $\S$ be a bundle on $Y$. 
With the notation above (and setting $\nu_{t+1}=\8$), one has
for  $1 \leq k \leq t$: 

(i) 
If $m<q\nu_{k+1}^{max}$, then 
\[
\HH^0(Y_p, \S^q(m))=\HH^0(Y_p, \S_k^q(m))
\]

\ \hspace{3mm} In particular, if $m<q\nu_1^{max}$, then $\HH^0(Y_p,
\S^q(m))=0$.

(ii) 
If $q\nu_k^{min}+\theta<m$, then 
\[
\HH^1(Y_p, \S^q_k(m))=0
\]

(iii) One has 
\[
\sum_{m= \T q\nu_k^{max} \7}^{\4 q\nu_k^{min}+\theta \3} \h^1(Y_p, \S_k^q(m))
=O\left(\dfrac{q^2}{p}\right)
\]
\end{lemma}
In particular, setting $k=t$ and noting that $\S_t=\S$, one sees that (ii) and (iii) 
yield the following. 

\begin{cor}
\label{corollary2-3}
\[
\sum_{m=\T q\nu_t^{max} \7}^\8 \h^1(Y_p, \S^q(m))
=O\left(\dfrac{q^2}{p}\right)
\]
\end{cor}

\begin{proof}[Proof of Lemma \ref{lemma123}]

(i) 
Consider the exact sequence
\[
0 \longto \S_k^q(m) \longto \S^q(m) \longto \S^q/\S_k^q \ (m) \longto 0.
\]
When $m<q\nu_{k+1}^{max}=\frac{-\mu_{k+1}^{max}}{\deg Y_p}$, we have
\[
\mu_{max}(\S^q/\S_k^q \ (m))
=\mu_{max}(\S^q/\S_k^q)+m \deg Y_p
=\mu_{k+1}^{max}+m\deg Y_p<0
\]
where the second equality is due to the fact that the HN
filtration of $\S^q/S_k^q$ is obtained via quotients from the
portion of the HN filtration of $\S^q$ that contains $\S_k^q$.
Thus $\HH^0(Y_p, \S^q/\S_k^q \ (m))=0$, and the result follows from the
long exact sequence of cohomology.

(ii)
By Serre duality,
\[
\HH^1(Y_p,\S_k^q(m))\cong \HH^0(Y_p,\S_k^q(m)^\vee \otimes \omega_{Y_p})
\]
But when
$m>q\nu_k^{min}+\theta=\dfrac{-\mu_k^{min}+\deg \omega_{Y_p}}{\deg Y_p}$,
we have
\begin{align*}
\mu_{max}(\S_k^q(m)^\vee \otimes
\omega_{Y_p})&=-\mu_{min}(\S_k^q(m))+\mu(\omega_{Y_p})\\
&=-(\mu_k^{min}+m\deg Y_p)+\deg \omega_{Y_p}<0
\end{align*}
and so $\HH^0(Y_p, \S_k^q(m)^\vee \otimes \omega_{Y_p})=0$.

(iii)
Since for $p\gg 0$ the bundle $\S_k$ on $Y_p$ is the specialization
(reduction mod $p$) of the corresponding subbundle in the HN
filtration of the syzygy bundle in characteristic zero, there
exist integers $\alpha_1, \dots, \alpha_s$ (independent of $p$)
and surjections of sheaves on $Y_p$
\[
\overset{s}{\underset{j=1}{\bigoplus}} \O(\alpha_j) \longto \S_k \longto 0
\]
for all $p\gg 0$. Applying the Frobenius pullback $(F^*)^n$,
twisting by $m$, and taking cohomology yields surjections
\[
\overset{s}{\underset{j=1}{\bigoplus}}\HH^1(Y_p,\O(q\alpha_j+m)) 
\longto \HH^1(Y_p,\S_k^q(m)) \longto 0
\]
Therefore it is enough to show that for any fixed integer $\alpha$
\[
\sum_{m=\T q\nu_k^{max}\7}^{\4 q\nu_k^{min}+\theta \3} \h^1(Y_p, \O(q\alpha+m))
=O \left(\dfrac{q^2}{p}\right)
\]
Reindexing and setting $L_0=q\alpha + \T q\nu_k^{max}\7$ and $L_1=q\alpha
+ \4 q\nu_k^{min}+\theta\3$ yields the sum 
\[
\sum_{l=L_0}^{L_1} \h^1(Y_p,\O(l))
\]

For those $p$ for which  $L_0\geq0$, this sum is bounded 
by Remark~\ref{lemma0} below. 
So, we may assume that $L_0<0$.
In that case, Remark~\ref{lemma0} again yields that the sum of the
terms with $\l \geq 0$ is bounded independent of $p$, and so, setting $L=\min(L_1,-1)$, 
we get
\[
\sum_{l=L_0}^{L_1} \h^1(Y_p, \O(l))=\sum_{l=L_0}^{L} \h^1(Y_p, \O(l))+O(1)
\]
In this remaining range, $\h^0(Y_p,\O(l))=0$ since $l<0$ 
and so the Riemann-Roch theorem yields the sum
\[
\sum_{l=L_0}^{L} ( -l\deg Y-(1-g)) +O(1)
= -\dfrac{\deg Y}{2}(L-L_0+1)(L+L_0)-(1-g)(L-L_0+1)+O(1)
\]
where we have used the following summation formula
\[
\sum_{l=a}^b l=\dfrac{(b-a+1)(b+a)}{2} \text{\qquad for any \ } a \leq b \in \mathbb Z
\]
Now, since $\nu_k^{min}=\nu_k+O(\frac{1}{p})$ and
$\nu_k^{max}=\nu_k+O(\frac{1}{p})$ by Corollary~\ref{Trivedi-lemma-simple}, we have
\[
\abs{L+L_0}
\leq \abs{L_1}+ \abs{L_0} \leq 
\abs{q\alpha+q\nu_k^{min}+\theta}+\abs{q\alpha+q\nu_k^{max}}+2
=O(q)
\]
and more crucially
\begin{align*}
0\leq L-L_0+1\leq  L_1-L_0+1 
&= \4 q\nu_k^{min}+\theta \3- \T q\nu_k^{max}\7 +1    \\
&\leq q(\nu_k^{min}-\nu_k^{max})+\theta +1 
=O\!\left(\dfrac{q}{p}\right)
\end{align*}
Plugging these two estimates in the above yields the desired result.
\end{proof}

The following variation of Serre's Vanishing Theorem is used in the proof above. 

\begin{remark}
\label{lemma0}
Note that for a locally free sheaf $\F$ on our family $\proj R \rightarrow \spec {\mathbb Z}$ there exists an $M>0$ (independent of $p$) such that
\[
\HH^1(Y_p, \F_p(m))=0, \text{ for all } m \geq M
\] 
and
\[
\sum_{m=0}^M \h^1(Y_p, \F_p(m))= O(1)
\]
For the generic fiber $Y_{\mathbb Q}$ there exists such a bound by Serre vanishing (\cite[Theorem III.5.2]{Ha}). By semicontinuity (\cite[Theorem III.12.8]{Ha}) it follows that
$\HH^1(Y_p, \F_p(M))=0$ for almost all primes $p$, and by the surjections
$H^1(Y_p, \F_p(m)) \rightarrow H^1(Y_p, \F_p(m+1))$ this is also true for all larger twists. The second statement follows also from semicontinuity.
\end{remark}

As a first step, we now use the lemma above to prove

\begin{lemma}
\label{lemma4}
For any integer $k$ with $1\le k\le t-1$, let $R=\sum_{i=1}^k r_i$
and $D=\sum_{i=1}^k r_i\nu_i$. Then
\[
\sum_{m=\T q\nu_k^{max} \7 }^{\T q\nu_{k+1}^{max} \7 -1} \h^0(Y_p, \S^q(m))
=q^2\deg Y \left( \dfrac{R}{2}(\nu_{k+1}^2-\nu_k^2)
-D(\nu_{k+1}-\nu_k) \right)
+O\left(\dfrac{q^2}{p}\right)
\]
\end{lemma}

\begin{proof}
By Lemma~\ref{lemma123}(i), in this range for $m$, one has
$\h^0(Y_p, \S^q(m))=\h^0(Y_p, \S_k^q(m))$. 
Applying the Riemann-Roch theorem then gives
\[
\sum_{m=\T q\nu_k^{max} \7 }^{\T q\nu_{k+1}^{max} \7 -1} \h^0(Y_p, \S^q(m))
=\sum_{m=\T q\nu_k^{max} \7 }^{\T q\nu_{k+1}^{max} \7 -1} (\deg \S_k^q(m)
+(\rank \S_k^q)(1-g)+\h^1(Y_p,\S_k^q(m)))
\]
By parts (ii) and (iii) of Lemma~\ref{lemma123}, $\sum
\h^1(Y_p, \S_k^q(m))=O(\frac{q^2}{p})$. Also, since $\rank \S_k^q= \rank
\S_k$, one has $\sum (\rank \S_k^q) (1-g)=O(q)$. Furthermore, by
additivity of slopes on tensor products
\begin{align*}
\deg \S_k^q(m) &= \deg S_k^q+(\rank \S_k^q)(\deg \O(m))\\
                &=q \deg \S_k+(\rank \S_k)m \deg Y\\
                &=q\sum_{i=1}^k r_i \mu_i+ m \deg Y \sum_{i=1}^k
                r_i\\
                &=\deg Y \left(-q \sum_{i=1}^kr_i\nu_i+m\sum_{i=1}^k
                r_i\right)\\
                &=\deg Y (mR-qD)
\end{align*}

Therefore the sum becomes
\begin{align*}
& \sum_{m=\T q\nu_k^{max} \7 }^{\T q\nu_{k+1}^{max} \7 -1}
\deg Y (mR-qD)\ +\ O\left(\dfrac{q^2}{p}\right)
\\
&= \deg Y \left(\dfrac{R}{2}(\T q\nu_{k+1}^{max} \7 - \T q\nu_k^{max} \7)(
\T q\nu_{k+1}^{max} \7 + \T q\nu_k^{max} \7-1) -qD(\T q\nu_{k+1}^{max} \7 -
\T q\nu_k^{max} \7)\right)+ O\left(\dfrac{q^2}{p}\right)
\end{align*}
But $\nu_k^{max}=\nu_k+O(\frac{1}{p})$ for each $k$ by
Corollary~\ref{Trivedi-lemma-simple}, and so the sum indeed simplifies to
\[
\deg Y \left(\dfrac{R}{2} \,q^2(\nu_{k+1}^2-\nu_k^2)-Dq^2(\nu_{k+1}-\nu_k)\right)+O\left(\dfrac{q^2}{p}\right)
\]
as desired.
\end{proof}

\section{Main Result}
\label{mainresult}

Now we return to the basic setting of this paper described at the
start of Section~\ref{background}. 
Recall that pulling back the
exact sequence on $Y_p$ 
\[
0\lra \S \lra \bigoplus_{i=1}^s \O(-d_i) 
\overset{f_1, \dots, f_s} {\longto} \O \lra 0
\] 
along the absolute Frobenius morphism $n$ times 
(with a subsequent twist by $m\in \mathbb Z$) yields the 
long exact sequence of cohomology
\[
0\lra H^0(Y_p, \S^q(m)) \lra \bigoplus_{i=1}^s H^0(Y_p, \O(m-qd_i)) 
\overset{f_1^q, \dots, f_s^q} {\longto} H^0(Y_p, \O(m)) \lra \cdots
\] 
where $\S^q$ denotes the pullback bundle $(F^*)^n(\S) = \syz (f_1^q, \dots,
f_s^q)$.  When $R_p$ is normal, one has that $H^0(Y_p, \O(n))\cong R_n$
for all $n\in\mathbb N$, and so the cokernel of $f_1^q, \dots, f_s^q$ 
is precisely the $m$-th graded piece of $R_p/I_p^{[q]}$. 

For the proof of the main theorem, we will use the results from the
previous section to analyze the cohomologies of $\S^q(m)$. As for
the cohomologies of the twists of the structure sheaf, we need the following
ingredient. Note that although
the statement looks like that of Lemma 2.2 of \cite{Br}, that result
cannot be applied here: For one thing, $\nu_t^{max}$ is not a fixed
number, and, even more crucially, ours is an asymptotic statement as
$p \to \8$, not as $n \to \8$. Yet the proof is essentially the same,
with these modifications in mind.

\begin{lemma}
\label{lemma5}
\[
\sum_{m=0}^{\T q\nu_t^{max} \7 -1} \h^0(Y_p, \O(m))
=q^2\dfrac{\deg Y}{2}\nu_t^2+O\left(\dfrac{q^2}{p}\right)
\]
\[\sum_{m=0}^{\T q\nu_t^{max} \7 -1}  \h^0(Y_p, \O(m-qd_i))
=q^2\dfrac{\deg Y}{2}(\nu_t-d_i)^2+O\left(\dfrac{q^2}{p}\right)
\]
\end{lemma}

\begin{proof}
As in Section~\ref{bundles}, we assume that $p$ has been taken large enough 
so that the genus and degree of $Y_p$ equal those of $Y$, and 
we denote them by $g$ and $\deg Y$, respectively. 

We prove the second statement; the proof of the first is similar.
By the Riemann-Roch theorem, one has 

\begin{align*}
\sum_{m=0}^{\T q\nu_t^{max} \7 -1}  \h^0(\O(m-qd_i))
&=\sum_{m=qd_i}^{\T q\nu_t^{max} \7 -1}  \h^0(\O(m-qd_i))
\\
&=\sum_{m=qd_i}^{\T q\nu_t^{max} \7 -1} (m-qd_i)\deg Y +(1-g)
+\h^1(\O(m-qd_i))
\\
&=\sum_{l=0}^{\T q\nu_t^{max} \7 -qd_i -1}
(l\deg Y+(1-g)+\h^1(\O(l)))
\\
&=\dfrac{\deg Y}{2}(\T q\nu_t^{max} \7 -qd_i)(\T q\nu_t^{max} \7-qd_i-1)
\\
& \hspace{1.4cm} +(1-g)(\T q\nu_t^{max} \7
-qd_i)+\sum_{l=0}^{\T q\nu_t^{max} \7 -qd_i -1} \h^1(\O(l))
\end{align*} 
The last term is $O(1)$ by Remark~\ref{lemma0}. Furthermore, since 
\[
\T q\nu_t^{max} \7 = q\nu_t^{max} + O(1)=q\nu_t + O\left(q \cdot \dfrac{1}{p}\right)
\]
by Corollary~\ref{Trivedi-lemma-simple}, the second term is $O(q)$ and the first term
becomes
\[
q^2\dfrac{\deg Y}{2}(\nu_t-d_i)^2+O\left(\dfrac{q^2}{p}\right)
\]
as desired.
\end{proof}

We are now ready to compute the desired limit.

\begin{theorem} 
\label{maintheorem}
Let $R$ be a standard-graded flat domain over $\mathbb Z$ such that almost all fiber rings 
$R_p=R \otimes_{\mathbb Z}{\mathbb Z}/p{\mathbb Z}$ are geometrically normal 2-dimensional domains and let $I=(f_1, \ldots, f_s)$ be a homogeneous $R_+$-primary ideal. 
Set $r_k$ and $\nu_k$ to be the ranks and normalized slopes of the 
quotients in the HN filtration of the syzygy bundle over $Y=\proj R_{\mathbb Q}$. 
For any fixed integer $n \geq 1$, setting $q=p^n$, one has 
\[
\dfrac{\l(R_p/I_p^{[q]})}{q^2}
=\dfrac{\deg Y}{2} \left(\sum_{k=1}^t r_k\nu_k^2-\sum_{i=1}^sd_i^2 \right)
+O\left(\dfrac{1}{p}\right)
\]
where $R_p=R\otimes_{\Z} \Z/p\Z$, $I_p^{[q]}=(f_1^q, \ldots, f_s^q)R_p$.
\end{theorem}

\begin{proof} The long exact sequence of cohomology for the exact
sequence
\[
0
\lra \S^q(m) 
\lra \bigoplus_{i=1}^s \O(m-qd_i) 
\overset{f_1^q, \ldots, f_s^q} {\longto} \O(m) 
\lra 0
\] 
yields the containment
\[
\coker \ \HH^0(f_1^q, \ldots, f_s^q)=( R_p/I_p^{[q]} )_m \subseteq \HH^1(Y_p, \S^q(m)).
\]
Therefore by Corollary~\ref{corollary2-3}
\[
\l(R_p/I_p^{[q]})
=\sum_{m=0}^{\8} \l \left( (R_p/I_p^{[q]})_m \right)
=\sum_{m=0}^{\T q\nu_t^{max} \7 -1} \l \left((R_p/I_p^{[q]})_m \right) + O\left(\dfrac{q^2}{p}\right)
\]
The beginning of the long exact sequence then yields
\[
\l(R_p/I_p^{[q]})
=\sum_{m=0}^{\T q\nu_t^{max} \7 -1}\bigg(\h^0(\O(m))
-\sum_{i=1}^s \h^0(\O(m-qd_i))
+\h^0(\S^q(m))\bigg) + O\left(\dfrac{q^2}{p}\right)
\]
After changing the order of summation, one may apply
Lemma~\ref{lemma5} to get
\[
=q^2\dfrac{\deg Y}{2}(\nu_t^2-\sum_{i=0}^s(\nu_t-d_i)^2)
+\sum_{m=0}^{\T q\nu_t^{max} \7 -1}\h^0(\S^q(m))
+O\left(\dfrac{q^2}{p}\right)
\]
Plugging in the result of Lemma~\ref{lemma4}, using the fact that
$\h^0(\S^q(m))=0$ for $m<\T q\nu_1^{max} \7$ by Lemma~\ref{lemma123}(i),
and simplifying as in Theorem~3.6 of \cite{Br} yields the
desired result.
\end{proof}

This finally brings us to our main goal: The expression on the right hand side
of the equation in Theorem~\ref{maintheorem} is equal to the limit Hilbert-Kunz
multiplicity
\[
e_{\text{HK}}^{\8}(I,R) 
\stackrel{\textrm{def}}{=}  
\lim_{p \to \8} e_{\text{HK}}(I_p,R_p)
\] 
as proved by Trivedi in \cite{Tr2}.  Therefore, we obtain the following consequence. 

\begin{cor} 
\label{maincorollary} 
With the notation as above, for any fixed $n \geq 1$ one has 
\[
\frac{\l(R_p/I_p^{[p^n]})}{(p^n)^2}
 = e_{\text{HK}}^{\8}(I,R)  + O\left(\frac{1}{p}\right)
\] 
In particular, 
\[
e_{\text{HK}}^{\8}(I,R)
= \lim_{p\to \8} \frac{\l(R_p/I_p^{[p^n]})}{(p^n)^2}
\]
\end{cor}

In fact, Trivedi shows that for these rings 
\[
e_{\text{HK}}(I_p,R_p) = e_{\text{HK}}^{\8}(I,R) + O\left(\frac{1}{p}\right)
\]
It is interesting to note that the bound $O(\frac{1}{p})$ on the speed of
convergence is of the same order as in Trivedi's result.

\begin{example}
\label{ex-optimal}
The following example can be found in Monsky's paper \cite{Mo-trinomial}. 
For the ring $R=\mathbb Z/p\mathbb Z[x,y,z]/(x^4+y^4+z^4)$ and 
the homogeneous maximal ideal $I=(x,y,z)$, one has
\[
e_{\text{HK}}(I,R) = 
\begin{cases}
3 + \frac{1}{p^2} & p \equiv 3, 5 \mod 8\\
3  &  p \equiv 1,7 \mod 8
\end{cases}
\]
It is not clear whether all these results are
optimal since we have not been able to find an example with 
the slower convergence rate of $O(\frac{1}{p})$. 
See also Example~\ref{ex-Chang} for diagonal hypersurfaces. 
\end{example}

\section{Diagonal hypersurfaces}
\label{monsky}

Unpublished results of Gessel and Monsky \cite{Mo-unpub} show that
$e_{\textrm{HK}}^{\8}(\mathfrak m,R)$ exists also for 
any diagonal hypersurface over $\mathbb Z$, that is, a ring of the form
\[
R=\frac{\mathbb Z [x_1, \dots, x_s]}{(x_1^{d_1} + \cdots + x_s^{d_s})}
\] 
with respect to  the homogeneous ideal $\mathfrak m$ generated by the variables. 
In this section we show how the proof simultaneously gives an affirmative 
answer to the question in our introduction for these rings, 
i.e., that for any {\it fixed} $n \geq 1$
\[
e_{\text{HK}}^{\8}(\mathfrak m,R)
= \lim_{p\to \8} \frac{\l(R_p/\mathfrak m_p^{[p^n]})}{(p^n)^d}
\]
Furthermore, we then use these methods to provide examples to show that a
certain naive limit in characteristic zero analogous to the one used
in positive characteristic to define the HK multiplicity does 
not give the same answer in general.

{\bf{Affirmative answer for diagonal hypersurface rings}}

We repeat a small part of the arguments from \cite{Mo-unpub} here to show how it yields the
result above. It uses the machinery developed by Han and Monsky in
\cite{HM} for computing HK multiplicities of diagonal hypersurfaces in
positive characteristic. For the notation, we generally refer the
reader to their paper, although the necessities are repeated here.  
For positive integers $k_1, \dots, k_s$ and field $F=\mathbb Z/p\mathbb Z$ define
\begin{align*}
D_F(k_1, \dots, k_s) 
=& \dim_F 
F[x_1, \dots, x_{s-1}]/(x_1^{k_1}, \cdots, x_{s-1}^{k_{s-1}}, 
(x_1+ \cdots + x_{s-1})^{k_s})
\qquad \qquad 
\\
=& \dim_F 
F[x_1, \dots, x_s]/(x_1^{k_1}, \cdots, x_s^{k_s}, x_1+ \cdots + x_s)
\end{align*}
In \cite{Mo-unpub}, Gessel and Monsky show that, for any $p$ and $n$, there are inequalities 
\begin{equation}
\label{monsky-inequal1}
d_1\cdots d_s \frac{D_F(\4 \frac{p}{d_1} \3, \dots, \4 \frac{p}{d_s} \3)}{p^d}
\leq 
\frac{\l(R_p/\mathfrak m_p^{[p^n]})}{(p^n)^d}
\leq 
d_1\cdots d_s \frac{D_F(\4 \frac{p}{d_1} \3+1, \dots, \4 \frac{p}{d_s} \3+1)}{p^d}
\end{equation}
As the outside terms are independent of $n$, taking the limit as $n$
goes to infinity yields inequalities
\begin{equation}
\label{monsky-inequal2}
d_1\cdots d_s \frac{D_F(\4 \frac{p}{d_1} \3, \dots, \4 \frac{p}{d_s} \3)}{p^d}
\leq 
e_{\text{HK}}(\mathfrak m_p,R_p)
\leq 
d_1\cdots d_s \frac{D_F(\4 \frac{p}{d_1} \3+1, \dots, \4 \frac{p}{d_s} \3+1)}{p^d}
\end{equation}
they then prove that, as $p$ goes to infinity, the outside terms
both converge to the same limit, and in fact, both equal 
\[
g\left(\frac{1}{d_1}, \dots, \frac{1}{d_s}\right) + O\left(\frac{1}{p}\right) 
\]
for the function $g\colon [0,1]^s \to \mathbb R$ 
defined as follows: for any numbers $x_1, \dots, x_s \in [0,1]$, set 
\begin{equation}
\label{defn-g}
g(x_1, \dots, x_s) 
= 
\frac{1}{2^{s-1}(s\!-\!1)!}
\sum_{\lambda\in\mathbb Z}  g_{\lambda}(x_1, \ldots, x_s)
\end{equation}
where 
\begin{equation}
\label{defn-glambda}
g_\lambda(x_1, \dots, x_s) 
= \sum_{\epsilon_i = \pm 1 {\textrm{ and }} \sum \epsilon_i x_i \geq 2 \lambda}
\epsilon_1\cdots \epsilon_s(\epsilon_1x_1+\cdots + \epsilon_sx_s-2\lambda)^{s-1}
\end{equation}
Note that $g$ is well-defined since $g_\lambda=0$ for $|\lambda| \gg 0$. 
But then the middle terms in both inequalities (\ref{monsky-inequal1}) and 
(\ref{monsky-inequal2}) go to the same limit (at the same rate) as well. 

In summary, we arrive at the following conclusion.  

\begin{theorem}[Gessel-Monsky]
\label{Monskyresult}
For any diagonal hypersurface ring 
\[
R=\frac{\mathbb Z [x_1, \dots, x_s]}{(x_1^{d_1} + \cdots + x_s^{d_s})} 
\qquad 
d_i \geq 2 
{\textrm{ for all }} i 
\]
with homogeneous maximal ideal $\mathfrak m$ and any fixed $n$, one has 
\[
e_{\text{HK}}^{\8}(\mathfrak m,R)
= e_{\text{HK}}(\mathfrak m_p,R_p) + O\left(\frac{1}{p}\right)
= \frac{\l(R_p/\mathfrak m_p^{[p^n]})}{(p^n)^d} + O\left(\frac{1}{p}\right)
\]
Furthermore, 
\[
e_{\text{HK}}^{\8}(\mathfrak m,R)
= g\left(\frac{1}{d_1}, \dots, \frac{1}{d_s}\right) 
\]
where the function $g$ is defined as above in (\ref{defn-g}) and (\ref{defn-glambda}). 
\end{theorem}

Note that, as for the case of homogeneous coordinate rings over smooth curves
in the previous section (see Corollary~\ref{maincorollary} and the discussion after it), 
the bounds on the rates of convergence of the various quantities to
$e_{\text{HK}}^{\8}(\mathfrak m,R)$ are the same. We do not know in this case 
either whether the bound $O(\frac{1}{p})$ on the speed of convergence is optimal. 

\begin{example}
\label{ex-Chang}
The diagonal hypersurface ring in Example~\ref{ex-optimal} satisfies
\[
e_{\text{HK}}(I,R) = e_{\text{HK}}^{\infty}(I,R) + O\left(\frac{1}{p^2}\right) 
\]
The same is true of the following example worked out by Chang in 
\cite{Ch} and Gessel and Monsky in \cite{Mo-unpub} using the techniques from \cite{HM}. 
For the homogeneous maximal ideal $I=(w,x,y,z)$ in the 
ring $R=\mathbb Z/p\mathbb Z[w,x,y,z]/(w^4+x^4+y^4+z^4)$, one has
\[
e_{\text{HK}}(I,R) 
=  \frac{8}{3} \left(  \frac{2p^2 \pm 2p + 3}{2p^2 \pm 2p + 1}   \right) 
\]
according as $p \equiv 1(4)$ or $p \equiv 3(4)$. Therefore, one finds that 
\[
e_{\text{HK}}(I,R) 
= \frac{8}{3} + O\left(\frac{1}{p^2}\right) 
\]
We do not know an example with the slower converge rate of $O(\frac{1}{p})$. 

\end{example}

{\bf{Limits in characteristic zero}}

Now we turn to using the results of Gessel and Monsky to examine why a certain naive
limit in characteristic zero fails to give the same answer. Given a
local (or graded) ring $R$ of equicharacteristic zero with (graded) maximal ideal $\mathfrak
m$, it might be tempting (in analogy with the definition of HK
multiplicity in positive characteristic) to take a set of generators
$x_1, \dots, x_r$ of $\mathfrak m$ and to look at the following limit (if it exists)
\[
e_{\textrm{naive}}^\8=
\lim_{N\to\8} \frac{\l(R_{\mathbb Q}/(x_1^N, \dots, x_r^N))}{N^d}
\]
Unfortunately, this can depend on the choice of generators, see Example~\ref{ex-naive2}, and even for minimal generators it does not yield $e_{\text{HK}}^{\8}(\mathfrak m,R)$ in general, see Example~\ref{ex-naive1}. 
In fact, their unpublished work \cite{Mo-unpub} 
enables one to compute this limit as well for diagonal hypersurfaces. 
Indeed, if we set 
\[
R=\frac{\mathbb Z [x_1, \dots, x_s]}{(x_1^{d_1} + \cdots + x_s^{d_s})}
\]
then by Lemma 2.2 of \cite{Mo-unpub} in view of Theorem 2.14 of \cite{HM} 
for the generators $x_1, \dots, x_s$ this limit equals the $\lambda=0$ term of $g(\frac{1}{d_1}, \dots, \frac{1}{d_s})$, 
that is
\[
e_{\textrm{naive}}^\8=
\frac{1}{2^{s-1}(s\!-\!1)!} \,g_0
\]
Therefore, whenever there are nonzero $g_\lambda$ terms 
in $g(\frac{1}{d_1}, \dots, \frac{1}{d_s})$ for some $\lambda\neq 0$, 
one might have $e_{\textrm{naive}}^\8 \neq e_{\text{HK}}^{\8}(\mathfrak m,R)$ by 
Theorem~\ref{Monskyresult}. We give explicit examples below. 

We begin with an example in which a minimal set of generators  
is used for $\mathfrak m$ in computing $e_{\textrm{naive}}^\8$ and yet one still does 
not obtain $e_{\text{HK}}^{\8}(\mathfrak m,R)$ as the limit. 
This is the ``smallest" example of which we know. 

\begin{example}
\label{ex-naive1}
In the notation above, let $s=5$ and $d_i=2$ for all $i$, that is, 
take the ring 
\[
R=\mathbb Z[x_1, \ldots, x_5]/(x_1^2 + \cdots + x_5^2)
\]
Then, writing $g_\lambda$ for 
$g_\lambda(\frac{1}{2},\frac{1}{2},\frac{1}{2},\frac{1}{2},\frac{1}{2})$, 
we have $g_\lambda =  0$ whenever $|\lambda| \geq 2$ and 
\[
g_1 = g_{-1} = 
=  \Big(   \frac{1}{2} + \frac{1}{2} + \frac{1}{2} + \frac{1}{2} + \frac{1}{2} - 2   \Big)^4
= \Big(   \frac{1}{2}   \Big)^4
\]

\noindent 
Monsky's Theorem~\ref{Monskyresult} then yields
\[
e_{\text{HK}}^{\8}(\mathfrak m,R) 
= \frac{2}{2^4 4!} \Big(  g_0 + 2\Big(   \frac{1}{2}   \Big)^4   \Big)
\]
whereas
\[
e_{\textrm{naive}}^\8
=  \frac{2}{2^4 4!}    g_0  
\]
\end{example}

Now we present a simpler example using similar ideas. It has the drawback though 
that minimial generating sets were not used when computing $e_{\textrm{naive}}^\8$. 

\begin{example}
\label{ex-naive2}
In the notation above, let $s=3$ and $d_i=1$ for all $i$, that is, 
take the ring 
\[
R=\mathbb Z[x_1, x_2, x_3]/(x_1 + x_2 + x_3)
\]
Then Theorem~2.14 of \cite{HM} shows that 
$R_{\mathbb Q}/(x_1^N,x_2^N,x_3^N)$ has dimension equal to 
$\T \frac{3}{4} N^2 \7$. (Monsky pointed out to us that this can also 
be proved by a simple argument involving a matrix of binomial 
coefficients.) So $e_{\textrm{naive}}^\8 = \frac{3}{4}$.
But, as $R$ is isomorphic to the regular ring $\mathbb Z[x_1, x_2]$, we know that 
$e_{\text{HK}}^{\8}(\mathfrak m,R)  =  1$.
\end{example}

\begin{remark}
It is interesting to compare and contrast these examples to the 
one given by Buchweitz and Chen in \cite{BC}. 
In contrast to our discussion above in characteristic 0, 
their results show that {\bf in characteristic} $p$ the naive limit 
does not even necessarily exist, even for a fixed choice of generators of the 
homogeneous maximal ideal. Specifically, for the ring 
\[
R_p=\mathbb Z/p \mathbb Z[x_1, x_2, x_3]/(x_1 + x_2 + x_3)
\]
(namely the reduction to characteristic $p$ of the ring in 
Example~\ref{ex-naive2} above) 
they show that the limit 
\[
\lim_{N\to\8} \frac{\l(R_p/(x_1^N, x_2^N, x_3^N))}{N^2}
\] 
does not exist. Indeed for the subsequence $N=p^n$ 
the limit is just the HK multiplicity, which equals 1 since $R_p$ is regular, 
but for the subsequence $N=2p^n$ the limit turns out to equal $\frac{3}{4}$
by an elementary computation. 
\end{remark}

More generally, the study in characteristic $p$ of how the length of 
\[
F[x,y]/(f^i,g^j,h^k),
\] 
where $F$ is a field, depends on $i$, $j$ and $k$ when $f$, $g$ 
and $h$ are fixed was carried out by Teixeira in his thesis \cite{Te}; 
the answer involves ``$p$-fractals".



\end{document}